\documentclass{amsart}

\usepackage{amssymb}
\usepackage{amsthm}
\usepackage{amsmath}
\usepackage{nicefrac}

\usepackage{enumitem}

\usepackage{ mathrsfs }

\usepackage{tikz}
\usetikzlibrary{arrows,calc}

\usepackage[foot]{amsaddr}

\usepackage{hyperref}

\usepackage{mathrsfs}

\theoremstyle{plain}
\newtheorem{proposition}{Proposition}[section]
\newtheorem{theorem}[proposition]{Theorem}
\newtheorem{lemma}[proposition]{Lemma}
\newtheorem{corollary}[proposition]{Corollary}
\newtheorem{question}[proposition]{Question}
\theoremstyle{definition}

\newtheorem{definition}[proposition]{Definition}
\newtheorem{observation}[proposition]{Observation}
\theoremstyle{remark}
\newtheorem{remark}[proposition]{Remark}

\DeclareMathOperator{\Aut}{Aut}

\DeclareMathOperator{\supp}{supp}

\DeclareMathOperator{\Euc}{Euc}

\DeclareMathOperator{\dist}{d}
\DeclareMathOperator{\Leb}{Leb}

\DeclareMathOperator{\Cc}{\mathcal{C}}

\DeclareMathOperator{\Hc}{\mathcal{H}}

\DeclareMathOperator{\Lc}{\mathcal{L}}

\DeclareMathOperator{\Oc}{\mathcal{O}}

\DeclareMathOperator{\Bb}{\mathbb{B}}
\DeclareMathOperator{\Cb}{\mathbb{C}}
\DeclareMathOperator{\Db}{\mathbb{D}}

\DeclareMathOperator{\Nb}{\mathbb{N}}

\DeclareMathOperator{\Rb}{\mathbb{R}}

\DeclareMathOperator{\Zb}{\mathbb{Z}}

\DeclareMathOperator{\Bf}{\mathsf{K}}
\DeclareMathOperator{\Bfd}{\mathsf{k}}
\DeclareMathOperator{\Gf}{\mathsf{G}}

\DeclareMathOperator{\wsq}{\mathsf{wsq}}
\DeclareMathOperator{\sq}{\mathsf{sq}}
\DeclareMathOperator{\Levi}{\mathscr{L}}

\newcommand{\abs}[1]{\left|#1\right|}

\newcommand{\norm}[1]{\left\|#1\right\|}
\newcommand{\wt}[1]{\widetilde{#1}}

\newcommand{\ip}[1]{\left\langle #1\right\rangle}

\begin{document}

\title{Weakly holomorphic homogeneous  regular manifolds}
\author{Andrew Zimmer}\address{Department of Mathematics, University of Wisconsin, Madison, WI, USA}
\email{amzimmer2@wisc.edu}
\date{\today}
\subjclass[2010]{}

\begin{abstract} We introduce a class of complex manifolds which we call weakly holomorphic homogeneous regular manifolds (wHHR) manifolds.  As the name suggests, this class contains the so-called holomorphic homogeneous regular manifolds but also other classes of complex manifolds such as two dimensional finite type domains and simply connected K\"ahler manifolds with pinched negative sectional curvature. For wHHR Stein manifolds we prove that the Bergman and Kobayashi metrics are biLipschitz equivalent.   \end{abstract}

\maketitle

\section{Introduction}

The Kobayashi and Bergman metrics are in general not uniformly biLipschitz on a complex manifold (even for domains in complex Euclidean space). However, for certain special classes of complex manifolds they are, such as
\begin{enumerate}
\item finite type domains in $\Cb^2$~\cite{Catlin1989}, 
\item simply connected negatively curved complete K\"ahler manifolds~\cite{WY2020},
\item the holomorphic homogeneous regular (\textsf{HHR}) manifolds introduced by Liu--Sun--Yau~\cite{LSY2004a} and independently by Yeung~\cite{Y2009}, which include:
\begin{enumerate}
\item strongly pseudoconvex domains,
\item Kobayashi hyperbolic convex domains, and
\item Teichm\"uller spaces.
\end{enumerate}
\end{enumerate}

The purpose of this paper is to introduce a new class of non-compact complex manifolds which contains all of the above families and where the Kobayashi and Bergman metrics are uniformly biLipschitz.

We begin by recalling the definition of holomorphic homogeneous regular manifolds.

\begin{definition}\cite{LSY2004a,Y2009} A complex $m$-manifold $M$ is a \emph{holomorphic homogeneous regular (\textsf{HHR}) manifold}  if there exists $s>0$ such that: for every $\zeta \in M$ there is a holomorphic embedding $f: M \rightarrow \Cb^m$ with $f(\zeta)=0$ and
\begin{align*}
s \Bb^m \subset f(M) \subset \Bb^m,
\end{align*}
where $\Bb^m \subset \Cb^m$ is the unit ball. 
\end{definition}

\begin{remark}
\textsf{HHR} manifolds are sometimes called manifolds with the \emph{uniform squeezing property}, see for instance~\cite{Y2009}. 
\end{remark}

For a general complex manifold, it is very difficult to construct bounded holomorphic functions much less bounded embeddings into complex Euclidean space (see for instance Yau's long standing conjecture~\cite[Problem 38]{Yau1982}). Motivated by Sibony's~\cite{Sib1981} lower bound on the Kobayashi metric, we replace the existence of a bounded holomorphic embeddings of $M$ into $\Cb^m$ with the existence of a bounded plurisubharmonic function with large Levi form. 

In the following definition let $\Levi(f)$ denote the Levi form of a $\Cc^2$ function $f$ and let $g_{\Euc}$ denote the Euclidean metric on $\Cb^m$. 

\begin{definition}\label{defn:wHHR} A complex $m$-manifold $M$ is a \emph{weakly holomorphic homogeneous regular (\textsf{wHHR}) manifold} if  there exists $s > 0$ such that: for every $\zeta \in M$ there exist a holomorphic embedding $\Phi : \Bb^m \rightarrow M$ and a $\Cc^2$ plurisubharmonic function $\phi : M \rightarrow [0,1]$ where $\Phi(0)= \zeta$ and 
\begin{align*}
\Levi(\phi \circ \Phi) \geq s^2 g_{\Euc} \text{ on } \Bb^m. 
\end{align*}
If, in addition, $M \subset \Cb^m$ is a domain, then we call $M$ a \textsf{wHHR} domain.
\end{definition}

As the name suggests every \textsf{HHR} manifold is an \textsf{wHHR} manifold.

\begin{observation}\label{obs:HHR_implies_wHHR}
If $M$ is a \textsf{HHR} manifold, then $M$ is also a \textsf{wHHR} manifold. 
\end{observation}

Since the proof is short we include it here. 

\begin{proof} Suppose $f : M \rightarrow \Cb^m$ is a holomorphic embedding with $f(\zeta) = 0$ and $s\Bb^m \subset f(M) \subset \Bb^m$. Then define $\Phi : \Bb^m \rightarrow M$ by $\Phi(w) = f^{-1}(sw)$ and define $\phi: M \rightarrow [0,1]$ by $\phi(z) = \norm{f(z)}^2$. Then $\phi$ is a $\Cc^2$ plurisubharmonic function and $\phi \circ \Phi(w) = s^2\norm{w}^2$. So 
\begin{align*}
\Levi(\phi \circ \Phi) = s^2 g_{\Euc} 
\end{align*}
on $\Bb^m$. 
\end{proof}

Like \textsf{HHR} manifolds~\cite{LSY2004a,Y2009}, we will show that the Kobayashi distance is Cauchy complete on a \textsf{wHHR} manifold. 

\begin{theorem}[see Corollary~\ref{cor:kob is cc} below]\label{thm:Kob_intro} If $M$ is a \textsf{wHHR} manifold, then the Kobayashi distance on $M$ is Cauchy complete. \end{theorem}

Combining Theorem~\ref{thm:Kob_intro} with results of Royden~\cite[Corollary pg. 136]{R1971} and Wu~\cite[Theorem F]{W1967} we deduce that \textsf{wHHR} domains are taut and pseudoconvex. 

\begin{corollary}\label{cor:pseudoconvex} If $\Omega \subset \Cb^m$ is a \textsf{wHHR} domain, then $\Omega$ is taut and pseudoconvex. \end{corollary} 

We can also estimate the Bergman metric on \textsf{wHHR} manifolds, but require the additional assumption that the manifold is Stein. 

\begin{theorem}[see Theorem~\ref{thm:wHHR_comparision} below]\label{thm:wHHR_comparision in intro} If $M$ is a \textsf{wHHR} Stein manifold, then:
\begin{enumerate}
\item The Bergman metric on $M$ is a complete K\"ahler metric with bounded geometry (in the sense of Definition~\ref{defn:bounded geometry} below).
\item The Kobayashi and Bergman metrics on $M$ are uniformly biLipschitz. 
\end{enumerate}
\end{theorem}

Notice that a \textsf{wHHR} domain is always Stein by Corollary~\ref{cor:pseudoconvex}. We require that the manifold is Stein in Theorem~\ref{thm:wHHR_comparision in intro} so that we can solve the $\bar{\partial}$-equation in weighted $L^2$ spaces.

\subsection{A pointwise version} The \emph{squeezing function}, introduced by Deng--Guan--Zhang~\cite{DGZ2012}, on a complex $m$-manifold $M$ is defined by 
\begin{align*}
\sq_M(\zeta) := \sup\{ r : & \text{ there exists an holomorphic embedding } \\
& f: M \rightarrow \Bb^m \text{ with } f(\zeta)=0 \text{ and }r\Bb^m \subset f(M) \}.
\end{align*}
Notice that a manifold is \textsf{HHR} if and only if the squeezing function has a positive lower bound. Further, if $M$ admits no bounded embeddings, then $\sq_M \equiv -\infty$. 

The squeezing function can be useful in some arguments because it provides a localized version of the \textsf{HHR} condition, see for instance the proof of Theorem 1.3 in~\cite{Z2018} or the arguments in Section 4 of~\cite{AFGG2020}. Motivated by this utility, we make the following definition. 

\begin{definition}\label{defn:weak_squeezing_function} Given a complex $m$-manifold $M$ the \emph{weak squeezing function} 
\begin{align*}
\wsq_M : M \rightarrow [0,1]
\end{align*}
is defined as follows: $\wsq_M(\zeta)$ is the supremum over all $s \in [0,1]$ where there exist a holomorphic embedding $\Phi : \Bb^m \rightarrow M$ and a $\Cc^2$ function $\phi : M \rightarrow [0,1]$ where $\Phi(0)=\zeta$, $\phi(\zeta) = 0$, $\log \phi$ is plurisubharmonic, and 
\begin{align*}
\mathscr{L}( \phi \circ \Phi)\geq s^2 g_{\Euc} 
\end{align*}
 on $\Bb^m$.
\end{definition}

\begin{remark} The conditions on $\phi$ in Definition~\ref{defn:weak_squeezing_function} might seem stronger than the conditions in Definition~\ref{defn:wHHR}, however it is straightforward to modify the function $\phi$ in Definition~\ref{defn:wHHR} to satisfy the assumptions in Definition~\ref{defn:weak_squeezing_function} for a possibly smaller value of $s$. The additional assumptions on $\phi$ in Definition~\ref{defn:weak_squeezing_function} are motivated by Sibony's lower bound for the Kobayashi metric~\cite{Sib1981}.
\end{remark}

This function has the following basic properties. 

\begin{proposition}[see Propositions~\ref{prop:wsq_basic in paper 1} and~\ref{prop:wsq_basic in paper 2} below]\label{prop:wsq_basic} If $M$ is a complex manifold, then 
\begin{enumerate}
\item $\sq_M(z_0) \leq \wsq_M(z_0) \leq 1$ for all $z_0 \in M$. 
\item $M$ is a \textsf{wHHR} manifold if and only if $\wsq_M$ has a positive lower bound.
\end{enumerate}
\end{proposition}

We will prove the following pointwise version of part (2) in Theorem~\ref{thm:wHHR_comparision in intro}.  

\begin{theorem}[see Theorem~\ref{thm:bergman_metric_versus_kob} below]\label{thm:wsq_comparison} For every $s \in (0,1)$ and $m \in \Nb$ there exists $C =C(s,m)> 1$ such that: If $M$ is a Stein $m$-manifold, $z_0 \in M$, and $\wsq_M(z_0) \geq s$, then the Bergman and Kobayashi metrics are $C$-biLipschitz at $z_0$. 
\end{theorem} 

We also prove the following rigidity result.

\begin{proposition}[see Proposition~\ref{prop:rigidity in paper} below]\label{prop:rigidity} If $M$ is a taut complex manifold and $\wsq_M(z) =1$ for some $z \in M$, then $M$ is biholomorphic to the unit ball.  \end{proposition}

It would be interesting to know if the tautness assumption can be removed from Proposition~\ref{prop:rigidity} and more generally if the supremum in Definition~\ref{defn:weak_squeezing_function} is always realized. 

\subsection{The $\bar{\partial}$-Neumann operator} In the last part of the paper we will establish a necessary condition for the compactness of the $\bar{\partial}$-Neumann operator  on $(0,q)$-forms. 

Given a pseudoconvex $\Omega \subset \Cb^m$ and $1 \leq q \leq m$, let $L^2_{(0,q)}(\Omega)$ denotes the space of $(0,q)$-forms with square integrable coefficients and let $\bar{\partial}^*$ denote the $L^2$ adjoint of $\bar{\partial}$. The $\bar{\partial}$-Neumann operator $N_q : L^2_{(0,q)}(\Omega) \rightarrow L^2_{(0,q)}(\Omega)$ is the bounded inverse to the unbounded self-adjoint  surjective operator $\square := \bar{\partial}\bar{\partial}^*+ \bar{\partial}^*\bar{\partial}$ on $L^2_{(0,q)}(\Omega)$. These operators have been extensively studied and we refer the reader to~\cite{FK1972,Krantz1992,BS1999,CS2001, S2010} for details.

It is generally believed that analytic varieties in the boundary should be an obstruction to the $\bar{\partial}$-Neumann operator being compact (assuming the boundary of the domain is sufficiently regular), see the discussion in~\cite{FS2001}. For \textsf{wHHR} domains we can use the methods in~\cite{Z2021} to prove this is indeed the case. 

\begin{theorem}[see Theorem~\ref{thm:necessary in paper} below]\label{thm:necessary} Suppose that $\Omega \subset \Cb^m$ is a bounded \textsf{wHHR} domain with  $\Cc^0$ boundary. If $\partial \Omega$ contains a $q$-dimensional analytic variety, then $N_q: L^2_{(0,q)}(\Omega) \rightarrow L^2_{(0,q)}(\Omega)$ is not compact.
\end{theorem}

To be precise, we say
\begin{enumerate}
\item $\Omega$ has $\Cc^0$ boundary if for every point $x \in \partial\Omega$ there exists a neighborhood $U$ of $x$ and there exists a linear change of coordinates which makes $U \cap \partial \Omega$ the graph of a $\Cc^{0}$ function. 
\item $\partial \Omega$ contains a $q$-dimensional analytic variety if there exists a holomorphic map $\psi : \Bb^q \rightarrow \Cb^m$ where $\psi(\Bb^q)\subset \partial \Omega$ and $\psi^\prime(0)$ has rank $q$. 
\end{enumerate} 

It would be interesting to know if the converse of Theorem~\ref{thm:necessary} is true for  bounded \textsf{wHHR} domains with  $\Cc^0$ boundaries. This would probably require developing a new way to show that the $\bar{\partial}$-Neumann operator is compact.

\subsection{Examples} By Observation~\ref{obs:HHR_implies_wHHR} every \textsf{HHR} manifold is also an \textsf{wHHR} manifold. Hence the following families are all \textsf{wHHR} manifolds:
\begin{enumerate}
\item The Teichm\"uller space of hyperbolic surfaces with genus $g$ and $n$ punctures (by the Bers embedding, see~\cite{Gard1987}).
\item Kobayashi hyperbolic convex domains or more generally $\Cb$-convex domains \cite{F1991, KZ2016, NA2017}.
\item Bounded domains where $\Aut(\Omega)$ acts co-compactly on $\Omega$.
\item Strongly pseudoconvex domains~\cite{DFW2014,DGZ2016}.
\end{enumerate}

In~\cite{Z2021}, we introduced a class of bounded domains called \emph{domains with bounded intrinsic geometry}. By ~\cite[Theorem 1.12]{Z2021} any domain with bounded intrinsic geometry is an \textsf{wHHR} domain. Hence, by the discussion in Section 2 in~\cite{Z2021} the following families are all \textsf{wHHR} manifolds:
\begin{enumerate}\setcounter{enumi}{4}
\item Finite type domains in $\Cb^2$.
\item Simply connected negatively curved complete K\"ahler manifolds.
\end{enumerate} 
Showing that  finite type domains in $\Cb^2$ have bounded instrinsic geometry uses deep work of Catlin~\cite{Catlin1989} and showing that simply connected negatively curved complete K\"ahler manifolds have bounded instrinsic geometry uses deep work of Wu--Yau~\cite{WY2020}. 

\subsection{Some questions} We end this introduction by listing some questions. 

\begin{question} Is every \textsf{wHHR} manifold Stein? 
\end{question} 

\begin{question} Is the supremum in Definition~\ref{defn:weak_squeezing_function} always realized?
\end{question} 

\begin{question} Does a \textsf{wHHR} (Stein?) manifold have a  complete K\"ahler--Einstein metric? If so, is it biLipschitz to the Bergman and Kobayashi metrics?
\end{question} 

\begin{question} Is it possible to characterize the finite type pseudoconvex domains which are \textsf{wHHR} domains? Are they the h-extendible finite type domains? \end{question} 

Griffiths~\cite{G1971} proved that if $X$ is a smooth quasi-projective algebraic variety and $x \in X$, then there exists a Zariski dense open set $\Oc$ of $X$ containing $x$ so that the universal cover $\widetilde \Oc$ of $\Oc$ is biholomorphic to a bounded domain in complex Euclidean space. 

\begin{question} Are the domains $\widetilde \Oc$ constructed by Griffiths \textsf{wHHR} or \textsf{HHR} domains? \end{question}

 \subsection*{Acknowledgements} This work was partially supported by a Sloan research fellowship and grant DMS-2105580 from the National Science Foundation.

\section{Preliminaries}

\subsection{Notations} In this section we fix any possibly ambiguous notation. 

\medskip

\noindent \emph{Approximate inequalities:} Given functions $f,h : X \rightarrow (0,\infty)$ we write $f \lesssim h$ or equivalently $h \gtrsim f$ if there exists a constant $C > 0$ such that $f(x) \leq C h(x)$ for all $x \in X$. Often times the set $X$ will be a set of parameters (e.g. $m \in \Nb$). 

\medskip

\noindent \emph{The Levi form: } Given a complex $m$-manifold $M$ and a $\Cc^2$-smooth real valued function $f : M \rightarrow \Rb$, the \emph{Levi form $\Levi(f)$ of $f$} is the possibly indefinite Hermitian form defined in local holomorphic coordinates $z^1,\dots,z^m$ by 
\begin{align*}
\Levi(f) := \sum_{1 \leq j,k \leq m} \frac{\partial^2 f}{\partial z^j \partial \bar{z}^k}dz^j \otimes d\bar{z}^k. 
\end{align*}

\subsection{Solutions to $\bar{\partial}$} Given a K\"ahler metric $g$ on a complex $m$-manifold $M$ let $dV_g$ denote the associated volume form. Recall that $g$ induces a norm on forms as follows: If $z^1, \dots, z^m$ are local coordinates centered at $z_0$ where $\frac{\partial}{\partial z^1}|_{z_0}, \dots, \frac{\partial}{\partial z^m}|_{z_0}$ are orthonormal with respect to $g$ and 
$$
\eta = \sum \alpha_{I,J} dz^I \wedge dz^J \in \Lambda^{(p,q)}_{z_0}(M), 
$$
then 
$$
\norm{\eta}_g = \sqrt{ \sum \abs{\alpha_{I,J}}^2 }.
$$
When $\eta$ is a $(m,0)$-form, we have the following formula (see~\cite[Section 3.2]{bern2010})
$$
i^{m^2} \eta \wedge \bar \eta = \norm{\eta}_g^2 dV_g. 
$$

We will use the following existence theorem for solutions to $\bar{\partial}$ which is established in~\cite[Theorem 8.8]{Dem1996} (similar versions appear in \cite[Proposition 2.1]{SY1977}, \cite[Proposition 8.6]{GW1979}, ~\cite[Proposition 4]{Ohsawa1984}, ~\cite[Theorem 5.1.1]{bern2010}).

\begin{theorem}\label{thm:existence} Suppose $M$ is a Stein $m$-manifold and $g$ is a (possibly incomplete) K\"ahler metric on $M$. Let $\lambda_1, \lambda_2$ be plurisubharmonic functions on $M$. Assume $\lambda_1$ is $\Cc^\infty$ smooth and 
\begin{align*}
\Levi(\lambda_1) \geq c g
\end{align*}
for some continuous positive function $c : M \rightarrow (0,\infty)$. 

If $f$ is a smooth $(m,1)$-form and $\bar{\partial}f= 0$, then there exists $F$ a smooth $(m,0)$-form with $\bar{\partial}F = f$ and 
\begin{align*}
i^{m^2}\int_M e^{-(\lambda_1+\lambda_2)} \left(F \wedge \bar{F} \right)   \leq \int_M \frac{\norm{f}_{g}^2}{c} e^{-(\lambda_1+\lambda_2)} dV_g,
\end{align*}
assuming the right hand side is finite. 
\end{theorem}

\section{The pluricomplex Green function}

In this section we establish local estimates for the pluricomplex Green function. We will use these estimates in Sections~\ref{sec:Kobayashi} and~\ref{sec:lower bounds on bergman stuff}.

\begin{definition} Suppose $M$ is a complex manifold and $\dist_M$ is some distance on $M$ induced by a complete Riemannian metric. The \emph{pluricomplex Green function} $\Gf_\Omega : \Omega \times \Omega \rightarrow \{-\infty\} \cup (-\infty, 0]$ is defined by 
\begin{align*}
\Gf_\Omega(z,z_0) = \sup \psi(z)
\end{align*}
where the supremum is taken over all negative plurisubharmonic functions $\psi$ such that $\psi(z)- \log \dist_M(z,z_0)$ is bounded from above in a neighborhood of $z_0$. 
\end{definition}

\begin{remark}
In the definition, we assume that $\psi \equiv -\infty$ is a plurisubharmonic function. 
\end{remark}

We will use the following fact.

\begin{proposition}\cite[Theorem 1.1]{K1985}\label{prop:Green basic facts} If $M_1,M_2$ are complex manifolds and $f : M_1 \rightarrow M_2$ is a holomorphic map, then
\begin{align*}
\Gf_{M_2}(f(z), f(z_0)) \leq \Gf_{M_1}(z,z_0)
\end{align*}
for all $z,z_0 \in M_1$. In particular, if $f$ is a biholomorphism, then $\Gf_{M_2}(f(z), f(z_0)) = \Gf_{M_1}(z,z_0)$ for all $z,z_0 \in M_1$. 
\end{proposition}

The main result in this section is the following.

\begin{theorem}\label{thm:green_fcn_bds}  For every $r,s \in (0,1)$ there exists $C=C(r,s) > 0$ such that: If $M$ is a complex $m$-manifold, $\Phi : \Bb^m \rightarrow M$ is a holomorphic embedding, $\phi : M \rightarrow [0,1]$ is a plurisubharmonic function with $\Lc( \phi \circ \Phi) \geq s^2 g_{\Euc}$ on $\Bb^m$, then 
\begin{align*}
\log \norm{w-w_0} - C \leq  \Gf_\Omega\left(\Phi(w),\Phi(w_0) \right) \leq \log \norm{w-w_0}+ C
\end{align*}
for all $w,w_0 \in r \Bb^m$.
\end{theorem}

\begin{proof} Fix a compactly supported smooth function $\chi : \Bb^m \rightarrow [0,1]$ with $\chi \equiv 1$ in a neighborhood of $r\Bb^m$. Then pick $A  > 0$ such that 
\begin{align*}
\Levi_w\left( \chi(w) \log \frac{\norm{w-w_0}}{2}\right) \geq -A g_{\Euc}
\end{align*}
when $w \in \Bb^m$ and $w_0 \in r\Bb^m$. We claim that $C := A/s^2+\log(2)$ satisfies the theorem. 

Fix $M$, $\Phi$, and $\phi$ as in the statement of the theorem. Then fix $w_0 \in r\Bb^m$. Define  $\psi : M \rightarrow \{-\infty\} \cup \Rb$ by 
\begin{align*}
\psi(z) = \begin{cases} \chi(\Phi^{-1}(z)) \log  \frac{\norm{\Phi^{-1}(z)-w_0}}{2} + \frac{A}{s^2} (\phi(z)-1) & \text{ if } z \in \Phi(\Bb^m) \\
 \frac{A}{s^2} (\phi(z)-1) & \text{ else}. 
\end{cases}
\end{align*}
Notice that $\psi$ is negative and $\psi\circ \Phi: \Bb^m \rightarrow [-\infty,0)$ is plurisubharmonic since 
\begin{align*}
\Levi( \psi \circ \Phi) \geq -A g_{\Euc} + \frac{A}{s^2}\Levi( \phi \circ \Phi) \geq 0. 
\end{align*}

Fix a distance $\dist_M$ on $M$ which is induced by a Riemannian metric. We claim that $\psi - \log \dist_M(\cdot, \Phi(w_0))$ is bounded from above in a neighborhood of $\Phi(w_0)$. Since $\Phi$ is an embedding, there exist $C' > 1$ and a neighborhood $W$ of $w_0$ such that 
$$
\frac{1}{C'} \norm{w-w_0} \leq \dist_M(\Phi(w), \Phi(w_0)) \leq C'\norm{w-w_0}
$$
for all $w \in W$. Then  $\psi - \log \dist_M(\cdot, \Phi(w_0))$ is bounded from above on $\Phi(W)$, which a neighborhood of $\Phi(w_0)$. Then, by the definition of $\Gf_M$,
\begin{align*}
\Gf_M(\Phi(w),\Phi(w_0)) \geq \psi(\Phi(w))  \geq \log \norm{w-w_0}-\frac{A}{s^2}.
\end{align*}
for all $w \in r\Bb^m$. 

For the upper bound, first notice that 
\begin{align*}
\Gf_{\Bb^m}(w,w_0) \leq \log \norm{w-w_0}-\log(2)
\end{align*}
since $w \mapsto \log \norm{w-w_0}-\log(2)$ is negative and plurisubharmonic. So by Proposition~\ref{prop:Green basic facts},
\begin{align*}
\Gf_M(\Phi(w),\Phi(w_0)) & \leq \Gf_{\Phi(\Bb^m)}(\Phi(w),\Phi(w_0)) =\Gf_{\Bb^m}(w,w_0) \\
& \leq \log \norm{w-w_0}-\log(2). \qedhere
\end{align*}
\end{proof}


\section{The Kobayashi metric}\label{sec:Kobayashi}


In this section we use the estimates on the pluricomplex Green function in Theorem~\ref{thm:green_fcn_bds} to bound the Kobayashi metric.

\begin{definition} Suppose $M$ is a complex manifold. The \emph{(infinitesimal) Kobayashi metric} is the pseudo-Finsler metric
\begin{align*}
k_{M}(z;X) = \inf \left\{ \abs{\xi} : \xi \in T_0\Cb \simeq \Cb, \ \varphi : \Db \rightarrow M \text{ holo.}, \ \varphi(0) = z, \ d(\varphi)_0\xi = X \right\}
\end{align*}
when $X \in T_z M$. The \emph{Kobayashi distance} is the pseudo-distance 
\begin{align*}
\dist_M^K(z_1,z_2) = \inf \left\{ \int_0^1 k_M(\sigma(t); \sigma^\prime(t)) dt : \begin{array}{c} \sigma:[0,1] \rightarrow M \text{ is a piecewise $\Cc^1$-smooth} \\  \text{curve joining $z_1$ to $z_2$} \end{array} \right\}. 
\end{align*}
\end{definition}

We will frequently use the following  fact.

\begin{observation} If $M_1,M_2$ are complex manifolds and $f : M_1 \rightarrow M_2$ is a holomorphic map, then
\begin{align*}
k_{M_2}(f(z); d(f)_zX) \leq k_{M_1}(z;X)
\end{align*}
for all $z \in M_1$ and $X \in T_z M_1$.
\end{observation}

The main result in this section is the following.

\begin{theorem}\label{thm:comp_to_kob} For every $r,s \in (0,1)$ there exists $C=C(r,s) > 1$ such that: If $M$ is a complex $m$-manifold, $\Phi : \Bb^m \rightarrow M$ is a holomorphic embedding, $\phi : M \rightarrow [0,1]$ is a plurisubharmonic function with $\Levi( \phi \circ \Phi) \geq s^2 g_{\Euc}$ on $\Bb^m$, then
\begin{align}\label{eqn:estimate on Kob metric}
\frac{1}{C} \norm{X} \leq  k_M\left(\Phi(w);d(\Phi)_wX\right) \leq C\norm{X}
\end{align}
for all $w \in r\Bb^m$ and $X \in T_w\Cb^m \simeq \Cb^m$. Moreover, if $w_1, w_2 \in r\Bb^m$, then
\begin{align}\label{eqn:estimate on Kob distance}
\frac{1}{C} \norm{w_1-w_2} \leq \dist_M^K(\Phi(w_1),\Phi(w_2)) \leq C \norm{w_1-w_2}.
\end{align}
\end{theorem}

Delaying the proof of Theorem~\ref{thm:comp_to_kob} we state and prove one corollary. 

\begin{corollary}\label{cor:kob is cc} If $M$ is a $\mathsf{wHHR}$ manifold, then the Kobayashi distance on $M$ is Cauchy complete. \end{corollary}

\begin{proof} Fix $s \in (0,1)$ and a family of embeddings  $\{ \Phi_\zeta :\zeta \in M\}$ satisfying Definition~\ref{defn:wHHR}. Then fix $r \in (0,1)$, and a constant $C =C(r,s)> 1$ satisfying  Theorem~\ref{thm:comp_to_kob}.

Suppose $(z_n)_{n \geq 1}$ is a Cauchy sequence  in $(M,\dist_M^K)$. Then there exists $N \geq 1$ such that  
\begin{align*}
\dist_M^K(z_N, z_n) < \frac{r}{2C}
\end{align*}
for all $n \geq N$. Then by Equation~\eqref{eqn:estimate on Kob distance} and the definition of the Kobayashi distance, we must have $z_n \in \Phi_{z_N}(\frac{r}{2}\Bb^m)$ for all $n \geq N$. 

Then 
\begin{align*}
\norm{\Phi_{z_N}^{-1}(z_{n_1})-\Phi_{z_N}^{-1}(z_{n_2})} \leq C\dist_M^K(z_{n_1}, z_{n_2}) 
\end{align*}
when $n_1, n_2 \geq N$. So $\{ \Phi_{z_N}^{-1}(z_{n}) : n \geq N\}$ is a Cauchy sequence in $\Cb^m$. So 
\begin{align*}
w:=\lim_{n \rightarrow \infty} \Phi_{z_N}^{-1}(z_{n})  \in \frac{r}{2}\overline{\Bb^m}
\end{align*}
exists. Thus 
\begin{equation*}
\lim_{n \rightarrow \infty} z_{n} = \Phi_{z_N}(w). \qedhere
\end{equation*}
\end{proof}

\begin{proof}[Proof of Theorem~\ref{thm:comp_to_kob}] Fix $r \in (0,1)$. To establish the estimates in~\eqref{eqn:estimate on Kob distance} we will have to establish the estimates in~\eqref{eqn:estimate on Kob metric} for a larger ball. So fix $r_1 \in (r,1)$ and a constant $C_0=C_0(r_1,s)> 0$ satisfying  Theorem~\ref{thm:green_fcn_bds} for $r_1,s$. 

Fix $M$, $\Phi$, and $\phi$ as in the statement of the theorem. Let 
$$
C_1 := \max\left\{ e^{C_0}, \frac{1}{1-r_1} \right\}.
$$
We claim that  if $w \in r_1\Bb^m$ and $X \in T_w\Cb^m \simeq \Cb^m$, then
\begin{align}\label{eqn:bound on Kob metric in proof}
\frac{1}{C_1} \norm{X} \leq  k_M\left(\Phi(w);d(\Phi)_wX\right) \leq C_1\norm{X}.
\end{align}
For the upper bound on the Kobayashi metric, define $\phi : \Db \rightarrow M$ by 
\begin{align*}
\phi(\lambda) = \Phi\left(  w+(1-r_1)\frac{X}{\norm{X}} \lambda\right). 
\end{align*}
Then $\phi(0)=\Phi(w)$ and $d(\phi)_0 \frac{\norm{X}}{1-r_1} = d(\Phi)_wX$. So by definition
\begin{align*}
k_M\left(\Phi(w);d(\Phi)_wX\right) \leq \frac{1}{1-r_1}\norm{X} \leq C_1 \norm{X}. 
\end{align*} 
For the lower bound, fix $\epsilon > 0$. Then there exist $\xi \in T_0 \Cb \simeq \Cb$ and a holomorphic map $\varphi : \Db \rightarrow M$ with $\varphi(0)=\Phi(w)$, $d(\varphi)_0\xi = d(\Phi)_wX$, and
\begin{align*}
\abs{\xi} \leq \epsilon+k_M\left(\Phi(w);d(\Phi)_wX\right).
\end{align*}
Then fix $\delta > 0$ such that $\varphi(\delta \Db) \subset \Phi(r_1\Bb^m)$. Then for $\lambda \in \delta \Db$ we have 
\begin{align*}
\log\abs{\lambda}=\Gf_{\Db}(\lambda,0) \geq \Gf_M(\varphi(\lambda), \Phi(w)) \geq \log \norm{\Phi^{-1}(\varphi(\lambda))-w} - C_0.
\end{align*}
So $\norm{ (\Phi^{-1}\circ \varphi)(\lambda)-w} \leq e^{C_0} \abs{\lambda}$ when $\lambda \in \delta \Db$. Thus $\norm{ d(\Phi^{-1}\circ \varphi)_w 1} \leq e^{C_0}$. So 
\begin{align*}
\norm{X}
= \norm{ d(\Phi^{-1}\circ \varphi)_w\xi} \leq e^{C_0}\abs{\xi} \leq e^{C_0}\left( \epsilon+k_M\left(\Phi(w);d(\Phi)_wX\right) \right). 
\end{align*} 
Since $\epsilon > 0$ is arbitrary, then 
\begin{align*}
\frac{1}{C_1} \norm{X} \leq e^{-C_0} \norm{X} \leq k_M\left(\Phi(w);d(\Phi)_wX\right).
\end{align*}

Now we prove the ``moreover'' part of the theorem.  Fix $w_1, w_2 \in r\Bb^m$. Equation~\eqref{eqn:bound on Kob metric in proof} implies that
\begin{align*}
\dist_M^K(\Phi(w_1),\Phi(w_2)) \leq C_1 \norm{w_1-w_2}.
\end{align*}
For the lower bound, let $\sigma : [0,1] \rightarrow M$ be a piecewise $\Cc^1$-smooth curve with $\sigma(0)=\Phi(w_1)$ and $\sigma(1) = \Phi(w_2)$. If the image of $\sigma$ is contained in $\Phi(r_1\Bb^m)$, then
\begin{align*}
\int_0^1 k_M(\sigma(t); \sigma^\prime(t)) dt \geq \int_0^1 \frac{1}{C_1} \norm{ (\Phi^{-1} \circ \sigma)^\prime(t)} dt \geq \frac{1}{C_1} \norm{w_1-w_2}. 
\end{align*}
 If the image of $\sigma$ is not contained in $\Phi(r_1\Bb^m)$, then there exist sequences $(a_n)_{n \geq 1}$ and $(b_n)_{n \geq 1}$ such that: $\sigma([0,a_n] \cup [b_n,1]) \subset \Phi(r_1\Bb^m)$ and 
 \begin{align*}
 \lim_{n \rightarrow \infty} \norm{ (\Phi^{-1} \circ \sigma)(a_n)} = r_1 =  \lim_{n \rightarrow \infty} \norm{ (\Phi^{-1} \circ \sigma)(b_n)}.
 \end{align*}
 Then Equation~\eqref{eqn:bound on Kob metric in proof} implies that
 \begin{align*}
\int_0^1 k_M(\sigma(t); \sigma^\prime(t)) dt & \geq \limsup_{n \rightarrow \infty} \int_{[0,a_n] \cup [b_n,1]} \frac{1}{C_1} \norm{ (\Phi^{-1} \circ \sigma)^\prime(t)} dt \\
 & \geq  \limsup_{n \rightarrow \infty}\frac{1}{C_1}\left( \norm{(\Phi^{-1} \circ \sigma)(a_n)-w_1}+\norm{w_2-(\Phi^{-1} \circ \sigma)(b_n)} \right) \\
 & \geq \frac{2(r_1-r)}{C_1} \geq \frac{r_1-r}{C_1r} \norm{w_1-w_2}.
\end{align*}

So $C: = \max\left\{ C_1, \frac{C_1r}{r_1-r}\right\}$ satisfies the theorem.

\end{proof}


\section{The weak squeezing function}


In this section we prove Propositions~\ref{prop:wsq_basic} and \ref{prop:rigidity} from the introduction. We start by proving part two of Proposition~\ref{prop:wsq_basic}. 

\begin{proposition}\label{prop:wsq_basic in paper 1} A complex manifold $M$ is a \textsf{wHHR} manifold if and only if $\wsq_M$ has a positive lower bound.
\end{proposition} 

\begin{proof}
It is clear that if $\wsq_M$ has a positive lower bound, then $M$ is a \textsf{wHHR} manifold. So suppose that  $M$ is a \textsf{wHHR} $m$-manifold and $s > 0$ satisfies Definition~\ref{defn:wHHR}. Fix $\delta \in (0,1)$ and a compactly supported smooth function $\chi : \Bb^m \rightarrow [0,1]$ with $\chi \equiv 1$ on $\delta\Bb^m$. Then pick $A  > 1$ such that 
\begin{align*}
\Levi_w\left( 2\chi(w) \log \norm{w}\right) \geq -A g_{\Euc}
\end{align*}
when $w \in \Bb^m$. Let $\lambda := \frac{A}{s^2}$. We will show that $\wsq_M$ is bounded below by $\sqrt{\frac{\delta^2}{4 \lambda e^{\lambda}}}$.

Fix $z_0 \in M$. Then fix a holomorphic embedding $\Phi : \Bb^m \rightarrow M$ and a $\Cc^2$ plurisubharmonic function $\phi : M \rightarrow [0,1]$ such that $\Phi(0) = z_0$, $\phi(z_0) = 0$, and 
\begin{align*}
\mathscr{L}( \phi \circ \Phi)\geq s^2 g_{\Euc}
\end{align*}
on $\Bb^m$.

Let $\phi_1 : = e^{-1+\phi} : M \rightarrow [0,1]$. Notice that if $X = \sum_j x_j \frac{\partial}{\partial z^j} \in T^{(1,0)} \Bb$, then 
\begin{align*}
\mathscr{L}( \phi_1 \circ \Phi)(X,\bar X) = e^{-1+\phi \circ \Phi} \mathscr{L}( \phi \circ \Phi)+ e^{-1+\phi \circ \Phi} \abs{X( \phi \circ \Phi)}^2
\end{align*}
and hence 
\begin{align*}
 \abs{X( \phi_1 \circ \Phi)}^2 =e^{-2+2\phi \circ \Phi}  \abs{X( \phi \circ \Phi)}^2 \leq \mathscr{L}( \phi_1 \circ \Phi)(X,\bar X).
\end{align*}

Define $\phi_2 : M \rightarrow [0,+\infty)$ by 
\begin{align*}
\phi_2(z) = \begin{cases} \norm{\Phi^{-1}(z)}^{2 \chi(\Phi^{-1}(z))} e^{\lambda \phi_1(z)} & \text{ if } z \in \Phi(\Bb^m) \\
e^{\lambda \phi_1(z)} & \text{ else}. 
\end{cases}
\end{align*}
Then $\phi_2(z_0) = 0$ and $\log \phi_2$ is plurisubharmonic by our choice of $A$ and $\lambda$. We claim that 
$$
 \mathscr{L}(\phi_2 \circ \Phi) \geq \frac{1}{4 \lambda }. 
$$
on $\delta \Bb^m$. To that end, fix $w \in \delta \Bb^m$ and $X = \sum_j x_j \frac{\partial}{\partial z^j} \in T^{(1,0)}_w \Cb^m$ with $\norm{X}=1$ and let 
$$
L : = \mathscr{L}(\phi_2 \circ \Phi)_w(X, \bar X).
$$
Notice that 
$$
\phi_2 \circ \Phi(w) =  \norm{w}^2 e^{\lambda \phi_1 \circ \Phi(w)} 
$$
and so
\begin{align*}
L  & =e^{\lambda \phi_1 \circ \Phi} \left( \norm{X}^2 +2\lambda  {\rm Re}\left( X( \phi_1 \circ \Phi)\cdot\ip{w,X}\right)\right) +\norm{w}^2\mathscr{L}( e^{\lambda \phi_1 \circ \Phi})(X,\bar X) \\
& \geq e^{\lambda \phi_1 \circ \Phi} \left(1 -2\lambda\abs{X( \phi_1 \circ \Phi)}\norm{w} \right)+\norm{w}^2\mathscr{L}( e^{\lambda \phi_1 \circ \Phi})(X,\bar X).
\end{align*}
Further,
\begin{align*}
\mathscr{L}( e^{\lambda \phi_1 \circ \Phi})(X,\bar X) & = e^{\lambda \phi_1 \circ \Phi}\left( \lambda  \mathscr{L}( \phi_1 \circ \Phi)(X,\bar X) +\lambda^2 \abs{X( \phi_1 \circ \Phi)}^2 \right) \\
& \geq e^{\lambda \phi_1 \circ \Phi}\left( \lambda  \abs{X( \phi_1 \circ \Phi)}^2 + \lambda^2  \abs{X( \phi_1 \circ \Phi)}^2 \right). 
\end{align*} 
So 
\begin{align*}
L  &  \geq  \lambda \abs{X( \phi_1 \circ \Phi)}^2\norm{w}^2   +  \left(1 -\lambda\abs{X( \phi_1 \circ \Phi)}\norm{w}\right)^2.
\end{align*}
Now if $\lambda\abs{X( \phi_1 \circ \Phi)}\norm{w} \leq 1/2$, we have 
$$
L \geq 0 + \frac{1}{4} \geq \frac{1}{4 \lambda}
$$
and if $\lambda\abs{X( \phi_1 \circ \Phi)}\norm{w} \geq 1/2$, we have 
$$ 
L \geq \frac{1}{4 \lambda} + 0=\frac{1}{4 \lambda}.
$$
So the claim is true. 

Finally define $\tilde \Phi : \Bb^m \rightarrow M$ by $\tilde \Phi(w) = \Phi(\delta w)$ and $\tilde \phi : M \rightarrow [0,1]$ by $\tilde \phi(z) = e^{-\lambda } \phi_2(z) $. Then 
$$
\mathscr{L}(\tilde \phi \circ\tilde \Phi) \geq \frac{\delta^2}{4 \lambda e^{\lambda}}g_{\rm Euc}
$$
on $\Bb^m$. So $\wsq_M(z_0) \geq \sqrt{\frac{\delta^2}{4 \lambda e^{\lambda}}}$. 
\end{proof}

For part one of Proposition~\ref{prop:wsq_basic} and Proposition~\ref{prop:rigidity},  we will use Sibony's Schwarz lemma for subharmonic functions on the disk. 

\begin{theorem}[{Sibony~\cite[Proposition 1]{Sib1981}}]\label{thm:sibony}  Suppose $\psi : \Db \rightarrow [0,1]$ is $\Cc^2$ in a neighborhood of the origin, $\psi(0)=0$, and $\log \psi$ is subharmonic. Then 
\begin{enumerate}
\item $\psi(z) \leq \abs{z}^2$ on $\Db$ with equality at some point different from $0$ if and only if $\psi(z) = \abs{z}^2$ for every $z \in \Db$. .
\item $\Delta \psi(0) \leq 4$ with equality if and only if $\psi(z) = \abs{z}^2$ for every $z \in \Db$. 
\end{enumerate}
\end{theorem} 

Using Sibony's result, we verify Proposition~\ref{prop:wsq_basic} from the introduction.

\begin{proposition}\label{prop:wsq_basic in paper 2} If $M$ is a complex manifold, then 
$$
\sq_M(z_0) \leq \wsq_M(z_0) \leq 1.
$$
\end{proposition} 

\begin{proof} We first show that $\sq_M(z_0) \leq \wsq_M(z_0)$. Suppose $f : M \rightarrow \Cb^m$ is a holomorphic embedding with $f(z_0) = 0$ and $s\Bb^m \subset f(M) \subset \Bb^m$. Then define $\Phi : \Bb^m \rightarrow M$ by $\Phi(w) = f^{-1}(sw)$ and define $\phi: M \rightarrow [0,1]$ by $\phi(z) = \norm{f(z)}^2$. Then $\Phi(0)=z_0$, $\phi(z_0)=0$, $\phi$ is a $\Cc^2$ function, and $\log \phi$ is plurisubharmonic. Further, $\phi \circ \Phi(w) = s^2\norm{w}^2$ and so
\begin{align*}
\Levi(\phi \circ \Phi) = s^2 g_{\Euc} 
\end{align*}
on $\Bb^m$. Hence $\wsq_M(z_0) \geq s$, which implies that $\sq_M(z_0) \leq \wsq_M(z_0)$. 

Next we show that $\wsq_M(z_0) \leq 1$. Suppose $\Phi : \Bb^m \rightarrow M$ is a holomorphic embedding with $\Phi(0)=z_0$ and $\phi : M \rightarrow [0,1]$ is a  $\Cc^2$ plurisubharmonic function where $\phi(\zeta) = 0$, $\log \phi$ is plurisubharmonic,   and 
\begin{align*}
\mathscr{L}( \phi \circ \Phi)\geq s^2 g_{\Euc} 
\end{align*}
on $\Bb$.

Fix a unit vector $X \in \Cb^m$ and consider the function $ \psi : \Db \rightarrow [0,1]$ defined by $\psi(z) = (\phi \circ \Phi)(zX)$. Then  $\psi$ is $\Cc^2$, $\psi(0)=0$, $\log \psi$ is subharmonic, and 
\begin{align*}
\Delta \psi(0) = 4\mathscr{L}( \phi \circ \Phi)_0(X,\bar X) \geq 4s^2.
\end{align*}
So by Theorem~\ref{thm:sibony} we must have $s \leq 1$. 
\end{proof}

To prove Proposition~\ref{prop:rigidity}, we will use the following lemma (which is similar in statement and proof to~\cite[Proposition 3.6.2]{MV2015}). 

\begin{lemma}\label{lem:averages} If $\psi : \Db \rightarrow [0,1]$ is $\Cc^2$, $\Delta \psi \geq 4s^2$ on $\Db$, and $\psi(0) = 0$, then 
\begin{align*}
s^2 r^2 \leq \frac{1}{2\pi} \int_0^{2 \pi} \psi(r e^{i\theta}) d\theta 
\end{align*}
for all $r \in [0,1)$. 
\end{lemma}

\begin{proof} Define $f : [0,1) \rightarrow [0,1]$ by 
\begin{align*}
f(r) = \frac{1}{2\pi} \int_0^{2 \pi} \psi(r e^{i\theta}) d\theta.
\end{align*}
Notice that $f$ is $\Cc^2$ and  $f(0) =0$. Further
\begin{align*}
f^{\prime\prime}(r) + \frac{1}{r} f^\prime(r) &= \frac{1}{2\pi} \int_0^{2 \pi} \left( \frac{\partial^2}{\partial r^2}+ \frac{1}{r} \frac{\partial}{\partial r}\right)\psi(r e^{i\theta}) d\theta \\
&=  \frac{1}{2\pi} \int_0^{2 \pi} (\Delta \psi)(r e^{i\theta}) d\theta \geq 4s^2
\end{align*}
since $\Delta = \frac{\partial^2}{\partial r^2}+ \frac{1}{r} \frac{\partial}{\partial r} + \frac{1}{r^2} \frac{\partial^2}{\partial \theta^2}$ and $\int_0^{2\pi}  \frac{\partial^2 }{\partial \theta^2}\psi(r e^{i\theta}) d\theta=0$. So 
\begin{align*}
\frac{d}{dr} \left( r f^\prime(r) \right) \geq 4s^2 r.
\end{align*}
Then integrating we have $r f^\prime(r) \geq 2s^2 r^2$ and so 
\begin{equation*}
f(r) = \int_0^r f^\prime(t) dt \geq s^2 r^2. \qedhere
\end{equation*}

\end{proof}

\begin{proposition}\label{prop:rigidity in paper} If $M$ is taut and $\wsq_M(z_0) = 1$ for some $z_0 \in M$, then $M$ is biholomorphic to the ball. \end{proposition} 

\begin{proof}Suppose $\wsq_M(z_0) = 1$. Then there exist $s_n \nearrow 1$, a sequence of holomorphic embedding $\Phi_n : \Bb^m \rightarrow M$, and a sequence of $\Cc^2$  functions $\phi_n : M \rightarrow [0,1]$ where $\Phi_n(0)=z_0$, $\phi_n(z_0) = 0$, $\log \phi_n$ is plurisubharmonic, and
\begin{align*}
\Levi( \phi_n \circ \Phi_n)\geq s_n^2 g_{\Euc}
\end{align*}
on $\Bb^m$. Notice that $\phi_n=e^{\log \phi_n}$ is also plurisubharmonic.

Since $M$ is taut, by passing to a subsequence we can suppose that $\Phi_n$ converges to a holomorphic map $\Phi : \Bb^m \rightarrow M$ with $\Phi(0)=\zeta$. By Theorem~\ref{thm:comp_to_kob}, for every $r \in (0,1)$ there exists $C_r > 1$ such that 
\begin{align*}
\frac{1}{C_r} \norm{w-u} \leq \dist_\Omega^K(\Phi_n(w), \Phi_n(u)) \leq C_r \norm{w-u}
\end{align*}
for all $n \geq 1$ and $p,q \in r \Bb$. Thus $\Phi$ is injective. 

Since $\phi_n$ is a bounded sequence of plurisubharmonic functions, by passing to a further subsequence we can suppose that $\phi_n$ converges in $L^{1,{\rm loc}}(M)$ to a plurisubharmonic function $\phi : M \rightarrow [0,1]$, see ~\cite[page 229]{Hormander}. By Theorem~\ref{thm:sibony} 
\begin{align*}
\phi_n \circ \Phi_n(z) \leq \norm{z}^2
\end{align*}
and by Lemma~\ref{lem:averages}
\begin{align*}
s_n^2 r^2 \leq \frac{1}{\mu(\mathbb{S}^{m-1})}\int_{\mathbb{S}^{m-1}} (\phi_n \circ \Phi_n)(rz) d\mu(z) 
\end{align*}
where $\mathbb{S}^{m-1} \subset \Cb^m$ is the unit sphere and $\mu$ is the surface area measure on $\mathbb{S}^{m-1}$. So we must have $(\phi \circ \Phi)(z) = \norm{z}^2$ almost everywhere. Since $\phi$ is upper semicontinuous, this implies that $(\phi \circ \Phi)(z) \geq \norm{z}^2$ for all $z \in \Bb^m$. 

Next we claim that $\Phi$ is onto. Since $\Phi$ is injective, $\Phi$ is an immersion and hence $\Phi(\Bb^m)$ is open. So it is enough to show that $\Phi(\Bb^m)$ is closed. Suppose not. Then there exists $z \in \overline{\Phi(\Bb^m)} \setminus \Phi(\Bb^m)$. Then there exists a sequence $(w_j)_{j \geq 1}$ in $\Bb^m$ with  $z = \lim_{j \rightarrow \infty} \Phi(w_j)$ and $\lim_{j \rightarrow \infty} \norm{w_j} =1$. Then 
\begin{align*}
\phi(z) \geq \lim_{j \rightarrow \infty} (\phi \circ \Phi)(w_j) \geq \lim_{j \rightarrow \infty} \norm{w_j}^2 \geq 1.
\end{align*}
Then by the maximum principle, $\phi \equiv 1$ which is impossible. So $\Phi$ is onto. 

So $\Phi : \Bb^m \rightarrow M$, being one-to-one and onto, is a biholomorphism. 
\end{proof}


\section{The Bergman kernel and metric on manifolds} 


In this expository section we recall the definition and basic properties of the Bergman  kernel, metric, and distance on a complex manifold. 

Given a complex $m$-manifold $M$, let $\Hc(M)$ denote the Hilbert space of holomorphic $(m,0)$-forms with inner product 
\begin{align*}
\ip{f,g} = \frac{i^{m^2}}{2^m} \int_M f \wedge \bar{g}. 
\end{align*}
The \emph{Bergman kernel of $M$} is the $(m,m)$-form on $M \times M$ defined by 
\begin{align*}
\Bf_M(z,w) = \sum_{j} \phi_j(z) \wedge \overline{\phi_j(w)}
\end{align*}
where $\{ \phi_j \}$ is some (any) orthonormal basis of $\Hc(M)$. We will also consider the $(m,m)$-form on $M$ defined by
\begin{align}\label{eqn:diag BK}
\Bfd_M(z) = \sum_{j} \phi_j(z) \wedge \overline{\phi_j(z)}
\end{align}
where again $\{ \phi_j \}$ is some (any) orthonormal basis of $\Hc(M)$.

When $\Bfd_M(z)$ is non-vanishing, the \emph{Bergman (pseudo-)metric} is defined in local holomorphic coordinates $z^1,\dots,z^m$ by 
\begin{align*}
g_{M} = \sum_{1 \leq j, k \leq m} \frac{\partial^2 \log\hat\Bfd(z)}{\partial z^j \partial \bar{z}^k} dz^j \otimes d\bar{z}^k
\end{align*}
where 
\begin{align*}
\Bfd_M  = \hat\Bfd \, dz^1 \wedge \dots \wedge dz^m \wedge d\bar{z}^1 \wedge \dots \wedge d\bar{z}^m.
\end{align*}

\subsection{Classical formulas for the Bergman metric and kernel}

In this section we recall and prove some classical formulas for the Bergman metric and kernel on a complex manifold. Fix a holomorphic embedding $\Phi : \Bb^m \rightarrow M$. Then define $\hat \Bfd : \Bb^m \rightarrow \Cb$ by 
\begin{align*}
\Phi^*\Bfd_M = \hat{\Bfd} \, dz^1 \wedge \dots \wedge dz^m \wedge d\bar{z}^1 \wedge \dots \wedge d\bar{z}^m.
\end{align*}
Likewise, for $f \in \Hc(M)$ define $\hat f : \Bb^m \rightarrow \Cb$ by 
$$
 \Phi^* f = \hat f dz^1 \wedge \cdots \wedge dz^m. 
 $$
 Notice that 
 \begin{align}
 \int_{\Bb^m} \abs{\hat f}^2 d\Leb & = \frac{i^{m^2}}{2^m}  \int_{\Bb^m}  \Phi^* f  \wedge \overline{ \Phi^* f } =  \frac{i^{m^2}}{2^m} \int_{\Phi(\Bb)} f \wedge \bar f   \leq  \norm{f}^2. \label{eqn:bound on L2 norm}
\end{align}

The following formulas as very well known, but since the argument is short we include it here. 

\begin{theorem}\label{thm:formulas for Bergman stuff} \
\begin{enumerate} 
\item If $z \in \Bb^m$, then 
$$
\hat \Bfd(z) = \max\left\{ \abs{\hat f(z)}^2 : f \in \Hc(M), \, \norm{f}= 1\right\}.
$$
\item Suppose $\Bfd_M$ is non-vanishing (and hence the Bergman metric exists). If $z \in \Bb^m$ and $X= \sum_{j=1}^m x_j \frac{\partial}{\partial z^j} \in T_z^{1,0} \Cb^m$, then 
$$
\Phi^* g_M(X,\bar X) = \frac{1}{\hat \Bfd(z)}  \max\left\{ \abs{X(\hat f)(z)}^2 : f \in \Hc(M), \, \norm{f} = 1, \, \hat f(z)=0\right\}.
$$
\end{enumerate}
\end{theorem} 

\begin{proof}[Proof sketch] Fix $z \in \Bb^m$ and $X= \sum_{j=1}^m x_j \frac{\partial}{\partial z^j} \in T_z^{1,0} \Cb^m$. 

Define $\ell_1 : \Hc(M) \rightarrow \Cb$ by 
$$
\ell_1(f) = \hat f(z).
$$
By Cauchy's integral formula and Equation~\eqref{eqn:bound on L2 norm}, $\ell_1$ is continuous. Hence there exists $\psi_1 \in \Hc(M)$ such that 
$$
\hat{f}(z) = \ell_1(f) = \ip{f,\psi_1}
$$ 
for all $f \in \Hc(M)$. Next define $\ell_2 : \psi_1^{\bot} \rightarrow \Cb$ by 
$$
\ell_2(f) = X(\hat f)(z).
$$
Again by Cauchy's integral formula and Equation~\eqref{eqn:bound on L2 norm}, $\ell_2$ is continuous. Hence there exists $\psi_2 \in \psi_1^{\bot}$ such that 
$$
X(\hat f)(z) = \ell_2(f) = \ip{f,\psi_2}
$$ 
for all $f \in \psi_1^{\bot}$. 

Let $\phi_1, \phi_2$ be orthogonal unit vectors such that $\psi_1 \in \Cb \cdot \phi_1$ and $\psi_2 \in \Cb \cdot \phi_2$ (it is possible for $\psi_1$ or $\psi_2$ to be zero). Notice that, if $f \in \psi_1^{\bot} $, then 
\begin{equation}\label{eqn:ip with psi1}
\hat{f}(z) = \ip{f,\psi_1}=0.
\end{equation}
Likewise, if $f \in \psi_1^{\bot} \cap \psi_2^{\bot}$, then 
\begin{equation}\label{eqn:ip with psi2}
X(\hat f)(z) = \ip{f,\psi_2}=0. 
\end{equation}

Now extend $\{\phi_1, \phi_2\}$ to an orthonormal basis $\{ \phi_j\}$. Then by Equation~\eqref{eqn:ip with psi1},
\begin{equation}\label{eqn:B and psi1}
\hat \Bfd(z)  = \sum \abs{\hat\phi_j(z)}^2 = \abs{\hat\phi_1(z)}^2 = \abs{ \ip{\phi_1, \psi_1}}^2 = \norm{\psi_1}^2. 
\end{equation}
Further, if $f =\sum c_j \phi_j \in \Hc(M)$, then 
$$
\abs{\hat{f}(z)} = \abs{\ip{f, \psi_1} }= \abs{c_1} \norm{\psi_1} \leq \norm{f} \norm{\psi_1}. 
$$
Thus (1) is true. 

Now suppose that $\hat \Bfd(z) \neq 0$. Since 
$$
\Phi^* g_M(X,\bar X)= X \bar X \left( \log \hat\Bfd \right)(z)= \frac{ X \bar X\left( \hat\Bfd\right)(z)}{\hat\Bfd(z)}-  \frac{ \bar X\left( \hat\Bfd\right)(z)}{\hat\Bfd(z)}\cdot\frac{ X \left( \hat\Bfd\right)(z)}{\hat\Bfd(z)}
$$
and $\hat\Bfd = \sum \abs{\hat \phi_j}^2$, Equations~\eqref{eqn:ip with psi1}, \eqref{eqn:ip with psi2}, and ~\eqref{eqn:B and psi1} imply
\begin{align*}
\Phi^* g_M(X,\bar X)& =\frac{1}{\hat\Bfd(z)} \sum \abs{X \left( \hat\phi_j\right)(z)}^2 - \frac{1}{\hat\Bfd(z)^2} \abs{\sum \overline{\hat\phi_j(z)} X \left( \hat \phi_j\right)(z) }^2 \\
& = \frac{1}{\hat\Bfd(z)}\left( \abs{X \left(  \hat\phi_1\right)(z)}^2+\abs{X \left( \hat \phi_2\right)(z)}^2 \right) -  \frac{1}{\hat\Bfd(z)^2}\abs{\overline{\hat\phi_1(z)} X \left( \hat \phi_1\right)(z)}^2 \\
& = \frac{1}{\hat\Bfd(z)}\abs{X \left( \hat \phi_2\right)(z)}^2=\frac{1}{\hat\Bfd(z)} \abs{ \ell_2( \phi_2)}^2 = \frac{1}{\hat\Bfd(z)} \abs{ \ip{\phi_2, \psi_2}}^2 =  \frac{1}{\hat\Bfd(z)} \norm{\psi_2}^2. 
\end{align*}
Further, if $f =\sum c_j \phi_j \in \Hc(M)$ and $\hat f(z) = 0$, then $f \in \psi_1^{\bot}$ and so  Equation~\eqref{eqn:ip with psi2} implies that
$$
 \abs{X(\hat f)(z)} = \abs{\ip{f, \psi_2} }= \abs{c_2} \norm{\psi_2} \leq \norm{f} \norm{\psi_2}. 
$$
So (2) is true. 
\end{proof}

\subsection{Bounds on the Bergman kernel in local coordinates}\label{sec:bounding bergman in local coordinates}

Results in~\cite{WY2020} provide uniform bounds on the Bergman kernel in local coordinates. 

\begin{proposition}\cite{WY2020} \label{prop:upper bounds on Bergman stuff} For every $m \in \Nb$ and $r \in (0,1)$ there exist constants $\{ C_{a,b}\}_{a,b \in \Zb^m_{\geq 0}}$  such that: If $M$ is a complex $m$-manifold, $\Phi : \Bb^m \rightarrow M$ is a holomorphic embedding, and $\beta : \Bb^m \times \Bb^m \rightarrow \Cb$ is defined by 
$$
(\Phi \times \Phi)^* \Bf_M = \beta(u,w)  du^1 \wedge \dots \wedge du^m \wedge d\bar{w}^1 \wedge \dots \wedge d\bar{w}^m,
$$
then
\begin{align*}
\abs{\frac{\partial^{\abs{a}+\abs{b}}\beta}{\partial u^{a}\partial \bar{w}^b}(u,w)} \leq C_{a,b}
\end{align*}
for all $u,w \in r\Bb$.  

\end{proposition} 

To prove the proposition we use the following lemma from ~\cite{WY2020}.

\begin{lemma}\cite[Corollary 24]{WY2020}\label{lem:WY lemma} Let $\Omega \subset \Cb^m$ be domain. Let $(h_j)_{j =1}^\infty$ be a sequence of holomorphic functions on $\Omega$ with the following property: There is an integer $N_0 \geq 0$ such that for all $N \geq N_0$, 
$$
\int_\Omega \abs{ \sum_{j=1}^N c_j h_j(z)}^2d\Leb \leq \sum_{j=1}^N \abs{c_j}^2 \quad \text{for all} \quad c_1,\dots, c_N \in \Cb.
$$
Then the series 
$$
H(z,w) = \sum_{j=1}^\infty h_j(z) \overline{h_j(w)} 
$$
converges uniformly and absolutely on every compact subset $\Omega \times \Omega$. Furthermore, there exists $C=C(m) > 0$ such that for every compact subset $E \subset \Omega$, 
$$
\max_{(z,w) \in E\times E} \abs{\frac{\partial^{\abs{a}+\abs{b}}H}{\partial z^{a}\partial \bar{w}^b}} \leq \frac{Ca!b!}{\dist_{\Cb^m}(E,\partial\Omega)^{2m+\abs{a}+\abs{b}}}
$$
for all multi-indices $a,b$. 
\end{lemma} 

\begin{proof}[Proof of Proposition~\ref{prop:upper bounds on Bergman stuff}]  We argue as in the proof of~\cite[Lemma 26]{WY2020}. Fix an orthonormal basis $\{\phi_j\}_{j \in J}$ of $\Hc(M)$. Define 
 $h_j : \Bb^m \rightarrow \Cb$ by 
$$
\Phi^* \phi_j = h_j dw^1 \wedge \dots \wedge dw^m.
$$
Then 
$$
\beta(u,w) = \sum_j h_j(u) \overline{h_j(w)}.
$$
Further, if $J' \subset J$ is finite and $(c_j)_{j \in J'}$ are complex numbers, then applying Equation~\eqref{eqn:bound on L2 norm} to $f := \sum_{j \in J'} c_j \phi_j$, we have 
\begin{align*}
\int_{\Bb^m} \abs{ \sum_{j \in J'} c_j h_j(w)}^2d\Leb \leq \norm{f}^2 =\sum_{j\in J'} \abs{c_j}^2.
 \end{align*} 
So the proposition follows from Lemma~\ref{lem:WY lemma}.
\end{proof} 

\subsection{The Bergman kernel on a domain} We now specialize to the case where $M = \Omega \subset \Cb^m$ is a domain. In this case, the Bergman kernel is often defined using holomorphic functions instead of forms. In particular, let $A^2(\Omega)$ denote the space of holomorphic functions $f : \Omega \rightarrow \Cb$ which are square integrable with respect to the Lebesgue measure. Then the \emph{Bergman kernel} is defined by 
\begin{align*}
\Bf_\Omega(z,w) = \sum_{j} \phi_j(z) \overline{\phi_j(w)}
\end{align*}
where $\{ \phi_j\}$ is some (any) orthonormal basis of $A^2(\Omega)$. 

Arguing as in Equation~\eqref{eqn:bound on L2 norm}, the map
\begin{align*}
f \in A^2(\Omega) \mapsto f dz^1 \wedge \dots \wedge dz^m  \in \Hc(\Omega)
\end{align*}
is an isomorphism of Hilbert spaces and hence
\begin{align*}
\Bf_\Omega(z,w) dz^1 \wedge \dots \wedge dz^m \wedge d\bar{w}^1 \wedge \dots \wedge d\bar{w}^m
\end{align*}
coincides with the differential form definition of the Bergman kernel.


\section{Lower bounds on the Bergman kernel and metric}\label{sec:lower bounds on bergman stuff}


Given a complex manifold, recall that $g_M$ denotes the Bergman metric and $\Bfd_M$ denotes the diagonal Bergman kernel in Equation~\eqref{eqn:diag BK} on a complex $m$-manifold $M$. In this section we establish lower bounds on both for \textsf{wHHR} Stein manifolds.

\begin{theorem}\label{thm:lower bounds on Bergman stuff} For every $m \in \Nb$ and $s \in (0,1)$ there exists $\epsilon=\epsilon(m,s) > 0$ such that: If $M$ is a Stein $m$-manifold, $\Phi : \Bb^m \rightarrow M$ is a holomorphic embedding, $\phi : M \rightarrow [0,1]$ is a plurisubharmonic function with $\Levi( \phi \circ \Phi) \geq s^2 g_{\Euc}$ on $\Bb^m$, then: 
\begin{enumerate}
\item If $X = \sum_{j=1}^m x_j \frac{\partial}{\partial z^j} \in T_0^{1,0} \Cb^m$, then 
$$
g_M(d\Phi_0X, d\Phi_0\bar X) \geq \epsilon \norm{X}.
$$
\item $\hat \Bfd(0) \geq \epsilon$ where $\hat \Bfd : \Bb^m \rightarrow [0,\infty)$ is defined by 
$$
\Phi^* \Bfd_M = \hat \Bfd  dz^1 \wedge \dots \wedge dz^m \wedge d\bar{z}^1 \wedge \dots \wedge d\bar{z}^m.
$$
\end{enumerate} 

\end{theorem} 

\begin{remark}
The following argument is based on the proof of~\cite[Proposition 8.9]{GW1979} which itself is based on work of H\"ormander~\cite{H1965}. See also \cite[Section 6]{Catlin1989} and \cite[Theorem 3.4]{M1994}.
\end{remark}

Theorem~\ref{thm:lower bounds on Bergman stuff} is an immediate consequence of the following lemma and Theorem~\ref{thm:formulas for Bergman stuff}.

\begin{lemma}\label{thm:extension} For every $m,n \in \Nb$ and $s \in (0,1)$ there exists $C=C(m,n,s) > 0$ such that: If $M$ is a Stein $m$-manifold, $\Phi : \Bb^m \rightarrow M$ is a holomorphic embedding, $\phi : M \rightarrow [0,1]$ is a plurisubharmonic function with $\Levi( \phi \circ \Phi) \geq s^2 g_{\Euc}$ on $\Bb^m$, $f : \Bb^m \rightarrow \Cb$ is holomorphic, and $\int_{\Bb^m} \abs{f}^2 d\Leb< \infty$, then there exists a holomorphic $(m,0)$-form $F$ on $M$ where:
\begin{enumerate}
\item If $\Phi^* F = \hat{F} dw^1 \wedge \dots \wedge dw^m$, then 
\begin{align*}
\frac{\partial^{\abs{\alpha}} \hat{F}}{\partial w^{\alpha}}(0) = \frac{\partial^{\abs{\alpha}} f}{\partial w^{\alpha}}(0) 
\end{align*}
for all multi-indices $\alpha$ with $\abs{\alpha} \leq n$.
\item $\norm{F}_{\Hc(M)}^2 \leq C\int_{\Bb^m} \abs{f}^2 d\Leb$.
\end{enumerate}
\end{lemma}

We will use the following observation. 

\begin{observation}[{see for instance ~\cite[Lemma 6.4]{bern2010}}]\label{obs:diff kahler and L2norms} If $g_1 \leq g_2$ are K\"ahler metrics on a complex $m$-manifold $M$ and $T$ is $(m,q)$-form, then 
$$
\norm{T}_{g_2}^2 dV_{g_2} \leq \norm{T}_{g_1}^2dV_{g_1}. 
$$
\end{observation} 

\begin{proof}[Proof of Lemma~\ref{thm:extension}]

Fix $m,n \in \Nb$ and $s \in (0,1)$. Also fix $\chi : \Bb^m \rightarrow [0,1]$, a compactly supported smooth function with $\chi \equiv 1$ on a neighborhood of $0$. 

Suppose $M$, $\Phi$, $\phi$, and $f$ satisfy the hypothesis of the theorem. Since $M$ is Stein, there exists a complete K\"ahler metric $g_0$ on $M$ and there exists a strictly plurisubharmonic function $\phi_0 : M \rightarrow [0,1]$. By scaling $g_0$ we can assume that 
$$ 
\Levi(\phi) > g_0
$$
on $\Phi(\supp(\chi))$. Then consider the K\"ahler metric 
$$
g := \Levi(\phi) + g_0.
$$

Let 
\begin{align*}
T & := \bar{\partial}(\Phi^{-1})^*\left( \chi f dw^1 \wedge \dots \wedge dw^m \right) =(\Phi^{-1})^*\bar{\partial}\left( \chi f dw^1 \wedge \dots \wedge dw^m \right) \\
& = (\Phi^{-1})^*\left( f \sum_{j=1}^m \frac{\partial \chi}{\partial w^j} d\bar{w}^j \wedge  dw^1 \wedge \dots \wedge dw^m \right).
\end{align*}
We will apply Theorem~\ref{thm:existence} to $\alpha$ with the metric $g$ and weights $\lambda_1 := \phi+\phi_0$, $\lambda_2 := 2(m+n)\Gf_M(\cdot, \Phi(0))$. Notice that 
$$
\Levi(\lambda_1) \geq \frac{1}{2} g
$$
on $\Phi(\supp(\chi))$.

Since $\Phi^* g \geq s^2 g_{\Euc}$, Observation~\ref{obs:diff kahler and L2norms} implies that 
\begin{align*}
\int_M \norm{T}_g^2 e^{-(\lambda_1+\lambda_2)} dV_g & =\int_{\Bb^m} \norm{\Phi^*T}_{\Phi^*g}^2 e^{-(\lambda_1+\lambda_2)\circ \Phi} dV_{\Phi^*g}  \\
&  \leq\int_{\Bb^m} \norm{\Phi^*T}_{s^2 g_{\Euc}}^2 e^{-(\lambda_1+\lambda_2)\circ \Phi}dV_{s^2 g_{\Euc}}  \\
& = \frac{1}{s^2}\int_{\Bb^m} \abs{ f}^2 \norm{ \bar{\partial}\chi}^2e^{-(\lambda_1+\lambda_2)\circ \Phi}d \Leb.
\end{align*}
Then, since $\bar{\partial}\chi \equiv 0$ on a neighborhood of $0$, Theorem~\ref{thm:green_fcn_bds} implies that there exists $C_1 > 0$ (which only depends on $m$, $n$, $s$, and $\chi$) such that 
\begin{align*}
\int_M \norm{T}_g^2 e^{-(\lambda_1+\lambda_2)} dV_g \leq C \int_{\Bb^m}\abs{ f}^2d \Leb. 
\end{align*}
By Theorem~\ref{thm:existence} there exist $C_1>0$ (which only depends on $m$, $n$, $s$, and $\chi$) and a smooth $(m,0)$-form $F_0$ such that $\bar{\partial}F_0 = T$ and 
\begin{align*}
\frac{i^{m^2}}{2^m} & \int_M e^{-(\lambda_1+\lambda_2)}\left(F_0 \wedge \bar{F_0}\right) \leq  C_1  \int_{\Bb^m} \abs{f}^2 d\Leb. 
\end{align*}
Let $\hat F_0 : \Bb^m \rightarrow \Cb$ be the function satisfying 
$$
\Phi^* F_0 = \hat F_0 dw^1 \wedge \cdots \wedge dw^m. 
$$
Then by Equation~\eqref{eqn:bound on L2 norm}, 
\begin{align*}
\int_{\Bb^m} \abs{\hat F_0}^2  e^{-\lambda_2 \circ \Phi} d\Leb \leq \frac{i^{m^2}}{2^m} \int_{M} e^{-\lambda_2} \left(F_0 \wedge \bar F_0 \right)<+\infty
\end{align*}
Further,  Theorem~\ref{thm:green_fcn_bds} implies that 
\begin{align*}
e^{-\lambda_2 \circ \Phi} = {\rm O}\left(\norm{z}^{-2(m+n)}\right).
\end{align*}
Thus, we must have 
\begin{align}\label{eqn:derivatives of F0}
\frac{\partial^{\abs{\alpha}} \hat F_0}{\partial w^{\alpha}}(0) =0 
\end{align}
 for all multi-indices $\alpha$ with $\abs{\alpha} \leq n$. 

Let 
\begin{align*}
F := (\Phi^{-1})^*(\chi f dw^1 \wedge \dots \wedge dw^m) - F_0.
\end{align*}
Then $\bar{\partial} F \equiv 0$ and so $F$ is holomorphic. Further, if $\Phi^* F = \hat{F} dw^1 \wedge \dots \wedge dw^m$, then Equation~\eqref{eqn:derivatives of F0} implies that 
\begin{align*}
\frac{\partial^{\abs{\alpha}} \hat{F}}{\partial w^{\alpha}}(0) = \frac{\partial^{\abs{\alpha}} (\chi f)}{\partial w^{\alpha}}(0) = \frac{\partial^{\abs{\alpha}} f}{\partial w^{\alpha}}(0) 
\end{align*}
for all multi-indices $\alpha$ with $\abs{\alpha} \leq n$. Finally, note that
\begin{align*}
\norm{F}_{\Hc(M)} &  \leq \norm{ (\Phi^{-1})^*(\chi f dw^1 \wedge \dots \wedge dw^m)}_{\Hc(M)} + \norm{F_0}_{\Hc(M)} \\
& \leq \norm{ f dw^1 \wedge \dots \wedge dw^m}_{\Hc(\Bb^m)} + \norm{F_0}_{\Hc(M)} \\
& \leq (1+\sqrt{C_1})  \left( \int_{\Bb^m} \abs{f}^2 d\Leb\right)^{1/2}
\end{align*}
and so the proof is complete.

\end{proof}


\section{The proofs of Theorems~\ref{thm:wHHR_comparision in intro} and~\ref{thm:wsq_comparison}}


In this section we prove Theorems~\ref{thm:wHHR_comparision in intro} and~\ref{thm:wsq_comparison} from the introduction. 

\begin{theorem}\label{thm:bergman_metric_versus_kob}  For every $s \in (0,1]$ and $m \in \Nb$ there exist constants $C> 1$, $\iota> 0$, and $(A_q)_{q \geq 0}$ such that: if $M$ is a Stein $m$-manifold, $z_0 \in M$, and $\wsq_M(z_0) \geq s$, then 
\begin{enumerate}
\item $\frac{1}{C} k_M(z_0;X) \leq \sqrt{ g_M(X,\bar X)} \leq C k_M(z_0;X)$ for all $X \in T_{z_0}^{(1,0)} M \simeq T_{z_0} M$. 
\item The injectivity radius of $g_M^B$ at $z_0$ is at least $\iota$,
\item $\norm{\nabla^q R|_{z_0}}_g \leq A_q$ where $R$ is the curvature tensor of $g_M$. 
\end{enumerate}
\end{theorem}

Delaying the proof for a moment, we state one definition and one corollary. 

\begin{definition}\label{defn:bounded geometry} A complete Riemannian manifold $(M,g)$ has \emph{bounded geometry} if the injectivity radius of $(M,g)$ is positive  and for every $q \geq 0$ we have 
$$
\sup_M \norm{\nabla^q R}_g < + \infty,
$$
 where $R$ is the curvature tensor of $g$. 
 \end{definition}
 
Since the Kobayashi metric induces a Cauchy complete distance on a \textsf{wHHR} manifold (see Corollary~\ref{cor:kob is cc}), we have the following corollary to Theorem~\ref{thm:bergman_metric_versus_kob}.

\begin{corollary}\label{thm:wHHR_comparision} If $M$ is a \textsf{wHHR} Stein manifold, then:
\begin{enumerate}
\item The Bergman metric on $M$ is a complete K\"ahler metric with bounded geometry. 
\item The Kobayashi and Bergman metrics on $M$ are uniformly biLipschitz. 
\end{enumerate}
\end{corollary}

\begin{proof}[Proof of Theorem~\ref{thm:bergman_metric_versus_kob}] Suppose $M$ is a Stein $m$-manifold and fix $z_0 \in M$ with $\wsq_M(z_0) \geq s$. By definition exists a holomorphic embedding $\Phi: \Bb^m \rightarrow M$ and a $\Cc^2$ plurisubharmonic function $\phi : M \rightarrow [0,1]$ such that $\Phi(0)=z_0$ and 
$$
\Levi( \phi \circ \Phi) \geq s^2 g_{\rm Euc}. 
$$

As in Section~\ref{sec:bounding bergman in local coordinates}, define $\beta :\Bb^m \times \Bb^m \rightarrow \Cb$ by
\begin{align*}
\beta(u,w) du^1 \wedge \dots \wedge du^m \wedge d\bar{w}^1 \wedge \dots \wedge d\bar{w}^m = (\Phi \times \Phi)^* \Bf_M.
\end{align*}
Then by definition, 
\begin{align*}
\Phi^* g_M = \sum_{1 \leq j,k \leq m} \frac{\partial^2 \log \beta(w,w)}{\partial w^j \partial \bar{w}^k}dw^j \otimes d\bar{w}^k.
\end{align*}

For $u \in \frac{1}{2} \Bb^m$, define $\Phi_u : \Bb^m \rightarrow M$ by $\Phi_u(w) =\Phi\left(u + \frac{1}{2} w\right)$. Then 
$$
\Levi( \phi \circ \Phi_u) \geq \left(\frac{s}{2}\right)^2 g_{\rm Euc}.
$$
So by Theorem~\ref{thm:lower bounds on Bergman stuff} there exists $\epsilon > 0$, which depends only on $s$ and $m$, such that 
\begin{equation}\label{eqn:lower bound on beta in some proof}
\beta(u,u) \geq \epsilon
\end{equation} 
for all $u \in \frac{1}{2}\Bb^m$ and 
$$
\epsilon g_{\rm Euc} \leq \Phi^*g_M
$$
on $ \frac{1}{2}\Bb^m$. 

Part (1) of Theorem~\ref{thm:bergman_metric_versus_kob} is a consequence of Theorem~\ref{thm:comp_to_kob} and the following lemma. 

\begin{lemma}\label{lem:comparable} There exists $C_1 > 1$, which depends only on $s$ and $m$, such that 
\begin{align*}
\frac{1}{C_1} g_{\rm Euc} \leq \Phi^*g_M \leq C_1 g_{\rm Euc}
\end{align*}
on $\frac{1}{2} \Bb$. 
\end{lemma}

\begin{proof} By Equation~\eqref{eqn:lower bound on beta in some proof}, the definition of $g_M$, and Proposition~\ref{prop:upper bounds on Bergman stuff} there exists $C_0 > 0$, which depends only on $s$ and $m$, such that 
\begin{align*}
 \Phi^*g_M \leq C_0 g_{\rm Euc}
\end{align*}
on $\frac{1}{2} \Bb$. Then let $C_1 : = \max\{ C_0, \epsilon^{-1}\}$. 
\end{proof}

\begin{lemma}\label{lem:curvature_bds} There exist constants $(A_q)_{q \geq 0}$, which depend only on $s$ and $m$, such that 
\begin{align*}
\sup_{z \in \Phi(\frac{1}{2} \Bb^m)} \norm{\nabla^m R}_{g_M}  \leq A_q. 
\end{align*}
\end{lemma}

\begin{proof}
This follows from Proposition~\ref{prop:upper bounds on Bergman stuff}, Lemma~\ref{lem:comparable}, and  expressing the curvature tensors in the  local coordinates induced by $\Phi$. 
\end{proof}

\begin{lemma} There exists $\iota > 0$, which depend only on $s$ and $m$, such that the injectivity radius of $g_M$ at $z_0$ is at least $\iota$.
\end{lemma}

\begin{proof} This follows immediately from~\cite[Proposition 2.1]{LSY2004b} and Lemmas~\ref{lem:comparable} and~\ref{lem:curvature_bds}.
\end{proof}

\end{proof} 


\section{Compactness of the $\bar{\partial}$-Neumann operator}\label{sec:compactness_of_dbar}


In this section we prove Theorem~\ref{thm:necessary} from the introduction. 

\begin{theorem}\label{thm:necessary in paper} Suppose that $\Omega \subset \Cb^m$ is a bounded \textsf{wHHR} domain with  $\Cc^0$ boundary. If $\partial \Omega$ contains a $q$-dimensional analytic variety, then $N_q: L^2_{(0,q)}(\Omega) \rightarrow L^2_{(0,q)}(\Omega)$ is not compact.
\end{theorem}

The proof is very similar to the proof of Proposition 11.6 in ~\cite{Z2021}, which in turn is similar to arguments of Catlin~\cite[Section 2]{C1983} and Fu--Straube~\cite[Section 4]{FS1998}. 

\subsection{Some notation}  Given a holomorphic map $\Phi : \Omega_1 \rightarrow \Omega_2$ between open sets $\Omega_j \subset \Cb^{m_j}$, we let $\Phi'(z)$ denote the $m_2$-by-$m_1$ complex Jacobian matrix and let $\norm{\Phi'(z)}$ denote the operator norm (relative to the Euclidean norms on $\Cb^{m_1}, \Cb^{m_2}$).

Given a $(p,q)$-form $\alpha = \sum \alpha_{I,J} dz^I \wedge d\bar{z}^J$ on a domain $\Omega$, we will let $\norm{\alpha}$ denote the function 
\begin{align*}
z \in \Omega \mapsto \left( \sum \abs{\alpha_{I,J}(z)}^2 \right)^{1/2}.
\end{align*}
and let 
$$
\norm{\alpha}_\Omega = \left(\int_\Omega \norm{\alpha}^2 d\Leb\right)^{1/2}.
$$
Similarly, we will let $\ip{\cdot, \cdot}$ denote the pointwise inner product on $(p,q)$-forms, that is
\begin{align*}
\ip{ \sum \alpha_{I,J} dz^I \wedge d\bar{z}^J, \sum \beta_{I,J} dz^I \wedge d\bar{z}^J} = \sum \alpha_{I,J}\bar{\beta}_{I,J}.
\end{align*}
Notice that $\norm{\alpha} = \sqrt{ \ip{\alpha,\alpha}}$. 

\subsection{Proof of Theorem~\ref{thm:necessary in paper}} Suppose $\partial \Omega$ that contains a $q$-dimensional analytic variety and suppose for a contradiction that $N_q: L^2_{(0,q)}(\Omega) \rightarrow L^2_{(0,q)}(\Omega)$ is compact. 

Define $S_q : L^2_{(0,q)}(\Omega) \cap \ker \bar{\partial} \rightarrow L^2_{(0,q-1)}(\Omega)$ by $S_q : = \bar{\partial}^* N_q$. Then $S_q$ is a \emph{solution operator for $\bar{\partial}$}, i.e. $\bar{\partial}S_q(f) = f$ for all $f \in L^2_{(0,q)}(\Omega) \cap \ker \bar{\partial}$. Further, since $N_q$ is compact so is $S_q=\bar{\partial}^* N_q$, see~\cite[Lemma 1]{FS2001}.

By assumption there exists an holomorphic map $\psi : \Bb^q \rightarrow \Cb^m$ where $\psi'(0)$ has rank $q$ and $\psi(\Bb^q) \subset \partial\Omega$. By rotating $\Omega$ we can assume that 
$$
\psi'(0)\Cb^q= \Cb^{q} \times \{0_{\Cb^{m-q}}\}. 
$$

Let $x:=\psi(0)$. Since $\Omega$ has $\Cc^0$ boundary, we can assume that there exist $\epsilon > 0$, $\delta > 0$, and a unit vector $v \in \Cb^m$ such that 
$$
x+tv+\psi(\epsilon\Bb^q) \subset \Omega
$$
for all $t \in (0,\delta]$. 

Fix $\{t_n\} \subset (0,\delta]$ converging to 0 and let $\zeta_n: = x+ t_n v$. Since $\Omega$ is a \textsf{wHHR} domain, there is some $s > 0$ such that for every $n \geq 1$ there exist a holomorphic embedding $\Phi_n : \Bb^m \rightarrow \Omega$ and a $\Cc^2$ function $\phi_n: \Omega \rightarrow [0,1]$ such that $\Phi_n(0)=\zeta_n$ and 
$$
\Levi( \phi_n \circ \Phi_n) \geq s^2 g_{\rm Euc}. 
$$
Precomposing each $\Phi_n$ with a rotation, we can assume that 
$$
\Phi_n'(0)(\Cb^{q} \times\{0_{\Cb^{m-q}}\})= \Cb^{q} \times\{0_{\Cb^{m-q}}\}=\psi'(0) \Cb^q.
$$

\begin{lemma}\label{lem:bounds on Phi_n at 0} There exists $C>1$ such that: If $n \geq 1$ and $X \in \Cb^{q} \times\{0_{\Cb^{m-q}}\}$, then 
$$
\frac{1}{C} \norm{X} \leq \norm{\Phi_n'(0) X} \leq C \norm{X}.
$$
\end{lemma}

\begin{proof} Since $\Omega$ is bounded, Cauchy's integral formulas imply that there exists $C_0 > 1$ such that 
$$
\norm{\Phi_n'(0) X} \leq C_0 \norm{X}.
$$
for all $n \geq 1$ and $X \in \Cb^m$. 

The other bound requires more work. Since $\psi'(0)$ is injective, there exists $\epsilon_1 > 0$ such that 
$$
\norm{\psi'(0)Y} \geq \epsilon_1\norm{Y}
$$
for all $Y \in \Cb^q$. By Theorem~\ref{thm:comp_to_kob}, there exists $\epsilon_2 > 0$ such that 
$$
k_\Omega(\zeta_n; \Phi_n'(0)X) \geq \epsilon_2\norm{X}
$$
for all $n \geq 1$ and $X \in \Cb^m$. 

Now fix $n \geq 1$ and $X \in \Cb^{q} \times\{0_{\Cb^{m-q}}\}$. Then there exists $Y \in \Cb^q$ with $\psi'(0)Y = \Phi_n'(0)X$. Since the Kobayashi metric is distance decreasing under holomorphic maps and 
$$
\zeta_n+\psi(\epsilon\Bb^q) \subset \Omega,
$$
we have 
$$
k_\Omega(\zeta_n;  \Phi_n'(0)X) = k_\Omega(\zeta_n;  \psi'(0)Y) \leq \frac{1}{\epsilon} \norm{Y}. 
$$ 
Thus 
$$
\norm{X} \leq \frac{1}{\epsilon\epsilon_2} \norm{Y} \leq \frac{1}{\epsilon \epsilon_1 \epsilon_2} \norm{\psi'(0)Y} = \frac{1}{\epsilon \epsilon_1 \epsilon_2} \norm{\Phi_n'(0)X}. 
$$

So $C : = \max\{ C_0, (\epsilon \epsilon_1 \epsilon_2)^{-1}\}$ suffices.  

\end{proof}

Consider the $(0,q)$-forms 
\begin{align*}
\alpha_n := \frac{\Bf_\Omega(\cdot, \zeta_n)}{\sqrt{\Bf_\Omega(\zeta_n, \zeta_n)}}d\bar{z}^1 \wedge \dots \wedge d\bar{z}^q
\end{align*}
on $\Omega$. Then $\norm{\alpha_n}_\Omega = 1$ and $\bar{\partial} \alpha_n =0$. Let $h_n := S_q(\alpha_n)$. Since $S_q$ is compact, after passing to a subsequence we can suppose that $h_n$ converges in $L^2_{(0,q-1)}(\Omega)$. Since $h_n$ converges,  for any $\epsilon > 0$ there exists a compact subset $K \subset \Omega$ such that 
\begin{align}
\label{eq:unif_estimate}
\sup_{n \geq 1} \int_{\Omega \setminus K} \norm{h_n}^2 d\Leb < \epsilon.
\end{align}

Define $\wt{\alpha}_n: \Bb^m \rightarrow \Cb$ by 
\begin{align*}
\wt{\alpha}_n :=  \det\left( \Phi_n^\prime(w) \right)\ip{ \Phi_n^* \alpha_n, d\bar{w}^1 \wedge \dots \wedge d\bar{w}^q}.
\end{align*}

\begin{lemma}\label{lem:alphas converge} After passing to a subsequence, we can assume that $\wt{\alpha}_n$ converges locally uniformly on $\Bb^m$ to a smooth function $\wt{\alpha}$ and $\wt{\alpha}(0) \neq 0$.
\end{lemma}

Delaying the proof of Lemma~\ref{lem:alphas converge} to the end of the section, we complete the proof of Theorem~\ref{thm:necessary in paper} by proving the following. 

\begin{lemma}\label{lem:mass_along_sequence} There exists a compact set $K \subset \Bb^m$ such that
\begin{align*}
\liminf_{n \rightarrow \infty} \int_{\Phi_n(K)} \norm{ h_n}^2 d\Leb > 0.
\end{align*}
Hence we have a contradiction with Equation~\eqref{eq:unif_estimate}.
\end{lemma}

\begin{proof} Since $\wt{\alpha} \neq 0$, there exists a smooth compactly supported function $\chi_0 : \Bb^m \rightarrow \Cb$ such that 
\begin{align*}
0 \neq \int_{ \Bb^m} \wt{\alpha}(w) \overline{\chi_0(w)} d\Leb.
\end{align*}
Since $\wt{\alpha}_n$ converges uniformly to $\wt{\alpha}$ on the support of $\chi_0$ we have 
\begin{align*}
0 \neq \int_{ \Bb^m} \wt{\alpha}(w) \overline{\chi_0(w)} d\Leb = \lim_{n \rightarrow \infty} \int_{ \Bb^m} \wt{\alpha}_n(w) \overline{\chi_0(w)} d\Leb.
\end{align*}
Then by the definition of $\wt{\alpha}_n$,
\begin{align*}
0 \neq \lim_{n \rightarrow \infty} \int_{ \Bb^m} \ip{  \det\left( \Phi_n^\prime(w) \right)\Phi_n^* \alpha_n, \chi }d\Leb
\end{align*}
where $\chi:=\chi_0 d\bar{w}^1 \wedge \dots \wedge d\bar{w}^q$. 

Since 
\begin{align*}
\det\left( \Phi_n^\prime(w) \right)\Phi_n^* \alpha_n & = \det\left( \Phi_n^\prime(w) \right)\Phi_n^*\bar{\partial} h_n =  \det\left( \Phi_n^\prime(w) \right)\bar{\partial} \Phi_n^*h_n \\
& = \bar{\partial} \left[ \det\left( \Phi_n^\prime(w) \right)\Phi_n^*h_n\right],
\end{align*}
we then have 
\begin{align*}
0 < &\lim_{n \rightarrow \infty}\abs{\int_{ \Bb^m} \ip{  \bar{\partial} \left[\det\left( \Phi_n^\prime(w) \right)\Phi_n^*h_n \right],\chi} d\Leb}  = \lim_{n \rightarrow \infty} \abs{\int_{ \Bb^m} \ip{  \det\left( \Phi_n^\prime(w) \right)\Phi_n^*h_n , \vartheta\chi} d\Leb} \\
& \lesssim \liminf_{n \rightarrow \infty} \left(\int_{ \supp(\chi)} \abs{\det\left( \Phi_n^\prime(w) \right)}^2\norm{\Phi_n^* h_n }^2d\Leb\right)^{1/2}
\end{align*}
where $\vartheta$ is the formal adjoint of $\bar{\partial}$. By Cauchy's integral formulas, 
$$
\max_{w \in \supp(\chi)} \norm{\Phi_n^\prime(w)} \lesssim 1. 
$$
So
\begin{align*}
\norm{\Phi_n^* h_n|_w } \leq \norm{\Phi_n^\prime(w)}^{q-1}\norm{h_n|_{\Phi_n(w)} } \lesssim \norm{h_n|_{\Phi_n(w)} }
\end{align*}
for $w \in \supp(\chi)$. Hence 
\begin{align*}
0 &<  \liminf_{n \rightarrow \infty} \int_{ \supp(\chi)} \abs{\det\left( \Phi_n^\prime(w) \right)}^2\norm{h_n|_{\Phi_n(w)} }^2d\Leb\\
& = \liminf_{n \rightarrow \infty} \int_{ \Phi_n(\supp(\chi))}  \norm{h_n}^2 d\Leb.
\end{align*}
Thus the first part of the lemma is true with $K : = \supp(\chi)$. 

Next we show that the first part of the lemma is incompatible with Equation~\eqref{eq:unif_estimate}. In particular, by Equation~\eqref{eq:unif_estimate}, we can fix a compact set $K' \subset \Omega$ such that 
$$
\sup_{n \geq 0} \int_{\Omega \setminus K'} \norm{h_n}^2 d\Leb < \liminf_{n \rightarrow \infty} \int_{\Phi_n(K)} \norm{ h_n}^2 d\Leb.
$$
However, by Corollary~\ref{cor:pseudoconvex}, $\Omega$ is pseudoconvex. Then since $\Phi_n(0) = \zeta_n \rightarrow x \in \partial \Omega$, we have 
$$
\Phi_n(K) \cap K' = \emptyset
$$ 
for $n$ large. So we have a contradiction.

\end{proof}

\subsection{Proof of Lemma~\ref{lem:alphas converge}} Define functions $f_n, J_n : \Bb^m \rightarrow \Cb$ by 
\begin{align*}
f_n(w) =\det\left( \Phi_n^\prime(w) \right) \frac{\Bf_\Omega(\Phi_n(w), \zeta_n)}{\sqrt{\Bf_\Omega(\zeta_n, \zeta_n)}} 
\end{align*}
and
\begin{align*}
J_n(w) =\ip{ \Phi_n^*(d\bar{z}^1 \wedge \dots \wedge d\bar{z}^q), d\bar{w}^1 \wedge \dots \wedge d\bar{w}^q}.
\end{align*}
Then 
\begin{align*}
\wt{\alpha}_n(w) = f_n(w) J_n(w).
\end{align*}

We will analyze $f_n$ and $J_n$ separately. 

\begin{lemma}\
\begin{enumerate}
\item For any $\delta \in (0,1)$, 
\begin{align*}
\sup_{n \geq 1} \sup_{w \in \delta \Bb^m} \abs{f_n(w)} <+\infty. 
\end{align*}
\item $\inf_{n \geq 1} \abs{f_n(0)} > 0$. 
\end{enumerate}

\end{lemma}

\begin{proof} Define $\beta_n : \Bb^m \times \Bb^m \rightarrow \Cb$ by 
$$
\beta_n(z,w) = \Bf_\Omega(\Phi_n(z), \Phi_n(w)) \det \Phi_n'(z)\overline{ \det \Phi_n'(w)}. 
$$
Notice that with $\Phi = \Phi_n$, $\beta_n$ coincides with the function $\beta$ appearing in Proposition~\ref{prop:upper bounds on Bergman stuff} and $\beta_n(0,0)$ coincides with $\hat\Bfd(0)$ 
appearing in Theorem~\ref{thm:lower bounds on Bergman stuff}. Thus 
$$
\inf_{n \geq 1} \abs{\beta_{n}(0,0)} > 0
$$ 
and for any $\delta \in (0,1)$, 
\begin{align*}
\sup_{n \geq 1} \sup_{w \in \delta \Bb^m} \abs{\beta_{n}(w,0)} <+\infty.
\end{align*}

Since 
$$
f_n(w) = \frac{\beta_{n}(w,0)}{\sqrt{\beta_{n}(0,0)}}\left( \frac{\overline{\det \Phi_n'(0)}}{ \abs{\det \Phi_n'(0)} }\right)^{-1},
$$
the lemma follows. 

\end{proof}

\begin{lemma}\label{lem:basic_estimate} \ 
\begin{enumerate}
\item $\overline{J_n(w)} = \det\left[ \frac{\partial( \Phi_{n})_j}{\partial w^k}(w)\right]_{1 \leq j,k \leq q}$, in particular each $J_n$ is anti-holomorphic. 
\item For any $\delta \in (0,1)$, 
\begin{align*}
\sup_{n \geq 1} \sup_{w \in \delta \Bb^m} \abs{J_n(w)} <+\infty. 
\end{align*}
\item $\inf_{n \geq 1} \abs{J_n(0)} > 0$. 
\end{enumerate}
\end{lemma}

\begin{proof} Notice that 
\begin{align*}
\Phi_{n}^* d\bar{z}^j = \sum_{k=1}^m \overline{\frac{\partial (\Phi_{n})_j}{\partial w^k}} d\bar{w}^k
\end{align*}
and so part (1) follows from the definition of the determinant. For part (2), Cauchy's integral formulas imply that each $\frac{\partial (\Phi_{n})_j}{\partial w^k}$ is uniformly bounded on $\delta \Bb^m$ and hence by part (1) so is $J_n$. For part (3), recall that $\Phi_n'(0)$ maps $\Cb^{q} \times \{0_{\Cb^{m-q}}\}$ to $\Cb^{q} \times \{0_{\Cb^{m-q}}\}$. Hence 
$$
\Phi_n'(0) = \begin{pmatrix} L_n & ? \\ 0 & ? \end{pmatrix} 
$$
where $L_n : = \left[ \frac{\partial( \Phi_{n})_j}{\partial w^k}(0)\right]_{1 \leq j,k \leq q}$. Then Lemma~\ref{lem:bounds on Phi_n at 0} implies that 
\begin{equation*}
\abs{J_n(0)} = \abs{\det L_n}\geq C^{-q}. \qedhere
\end{equation*}
\end{proof} 

Applying Montel's theorem to the sequences $(f_n)_{n \geq 1}$ and $(\bar J_n)_{n \geq 1}$, we can pass to a subsequence so that $\wt{\alpha}_n$ converges locally uniformly on $\Bb^m$ to a smooth function $\wt{\alpha}$ and 
$$
\abs{\wt{\alpha}(0)} \geq \left( \inf_{n \geq 1} \abs{f_n(0)} \right) \left( \inf_{n \geq 1} \abs{J_n(0)} \right) > 0. 
$$

\bibliographystyle{alpha}
\bibliography{complex}

\end{document}